\documentclass{amsart}

\usepackage[latin1]{inputenc}
\usepackage{amssymb,amsmath}
\usepackage[mathscr]{eucal}
\usepackage{enumerate}

\newtheorem{defi}{Definition}[section]
\newtheorem{teor}[defi]{Theorem}
\newtheorem{lema}[defi]{Lemma}
\newtheorem{nota}[defi]{Remark}
\newtheorem{coro}[defi]{Corollary}
\newtheorem{prop}[defi]{Proposition}

\newtheorem{ejem}[defi]{Example}

\newcommand{\e}{\varepsilon}

\newcommand{\M}{\mathcal{M}}
\newcommand{\MB}{\mathcal{BM}}
\newcommand{\MBR}{\mathcal{BM}_{(p_1,p_2,p_3)}(\R)}
\newcommand{\mul}{\mathcal{\tilde M}_{(p_1,p_2,p_3)}(\R)}

\newcommand{\C}{\mathbb{C}}

\newcommand{\N}{\mathbb{N}}
\newcommand{\R}{\mathbb R}

\newcommand{\ee}{\end{equation}}
\newcommand{\be}{\begin{equation}}
\newcommand{\ea}{\end{eqnarray*}}
\newcommand{\ba}{\begin{eqnarray*}}

\begin{document}

\def\dem{\noindent {\sf  Proof.\ \ }}
\def\qed{\hfill $\blacksquare$}

\textwidth 16cm  
\textheight 21cm \headsep 1.75cm \topmargin -1.5cm \oddsidemargin
0in \evensidemargin 0in

\title{Notes on the spaces of bilinear multipliers}

\author{Oscar Blasco}

\address{Department of Mathematics,
Universitat de Valencia, Burjassot 46100 (Valencia)
 Spain}
\email{oscar.blasco@uv.es}

\thanks{
Partially supported by  Proyecto MTM2008-04594/MTM} \maketitle

\begin{abstract} A   locally integrable function  $m(\xi,\eta)$ defined on $\R^n\times
\R^n$ is said to be a  bilinear multiplier on $\R^n$ of type
$(p_1,p_2, p_3)$ if
$$ B_m(f,g)(x)=\int_{\R^n} \int_{\R^n}\hat f(\xi)\hat g(\eta)m(\xi,\eta)e^{2\pi i(\langle\xi +\eta,x\rangle} d\xi d\eta$$
defines a bounded bilinear operator from $L^{p_1}(\R^n)\times
L^{p_2}(\R^n) $ to $L^{p_3}(\R^n)$. The  study of the basic
properties of such   spaces  is investigated and several methods
of constructing examples of bilinear multipliers are provided. The
special case where $m(\xi,\eta)= M(\xi-\eta)$ for a given $M$
defined on $\R^n$ is also addressed.
\end{abstract}

\section{Introduction.}

Throughout the paper $C_{00}(\R^n)$ denotes the space of
continuous functions defined in $\R^n$ with compact support,
${\mathcal S}(\R^n)$ denotes the Schwartz class on $\R^n$, i.e.
$f:\R^n\to \C$ such that $f\in C^\infty(\R^n)$ and $x^\alpha
\frac{\partial^{|\beta| }f(x)}{\partial x_1^{\beta_1}...\partial
x_n^{\beta_n}}$ is bounded for any $\beta=(\beta_1,...,\beta_n)$
and $\alpha=(\alpha_1,...,\alpha_n)$ where
$x^\alpha=x_1^{\alpha_1}...x_n^{\alpha_n}$ and $|\beta|=
\beta_1+...+\beta_n$ and ${\mathcal P}(\R^n)$ stands for the set
of functions in ${\mathcal S}(\R^n)$ such that $\hat f \in
C_{00}(\R^n)$ where  $\hat f(\xi)= \int_{\R^n} f(x) e^{-2\pi
i\langle x,\xi\rangle}dx$.

 We shall use the notation $\M_{p,q}(\R^n)$ (respect.
 $\tilde\M_{p,q}(\R^n)$), for
$1\le p,q\le \infty$, for the space of distributions $u\in
{\mathcal S}'(\R^n)$ such that $u*\phi \in L^q(\R^n)$ for all
$\phi\in L^p(\R^n)$ (respect. for the space of bounded functions
$m$ such that $T_m$ defines a bounded operator from $L^p(\R^n)$ to
$L^q(\R^n)$ where $\widehat{T_m (\phi)}(\xi)= m(\xi)\hat f(\xi)$.)
We endow the space $\tilde\M_{p,q}(\R^n)$ with the ``norm" of the
operator $T_m$, that is $\|m\|_{p,q}=\|T_m\|$.

Let us start off by mentioning some  well known properties of the
space of linear multipliers (see \cite{BL, SW}):  $\M_{p,q}(\R^n)=
\{0\}$ whenever $q<p$, $\M_{p,q}(\R^n)=\M_{q',p'}(\R^n)$ for
$1<p\le q<\infty$ and for $1\le p\le 2$, $\M_{1,1}(\R^n)\subset
\M_{p,p}(\R^n)\subset \M_{2,2}(\R^n).$ We also have the
identifications
$$\tilde\M_{2,2}(\R^n)=
L^\infty(\R^n),$$
$$\M_{1,q}(\R^n)=\{u\in
{\mathcal S}'(\R^n): u\in L^q(\R^n)\}, 1<q<\infty, $$
$$\M_{1,1}(\R^n)=\{u\in
{\mathcal S}'(\R^n): u= \mu \in M(\R^n)\}.$$

In this paper we shall be dealing with their bilinear analogues.
\begin{defi}  Let $1\le p_1,p_2\le\infty$ and $0<p_3\le \infty$ and let $m(\xi,\eta)$ be a locally integrable function on $\R^n\times \R^n$. Define
$$ B_m(f,g)(x)=\int_{\R^n} \int_{\R^n}\hat f(\xi)\hat g(\eta)m(\xi,\eta)e^{2\pi i(\langle\xi +\eta,x\rangle} d\xi d\eta$$
for $f, g\in {\mathcal P}(\R^n)$.

  $m$ is said to be a {\it bilinear
multiplier} on $\R^n$ of type $(p_1,p_2, p_3)$    if  there exists
$C>0$ such that $$\|B_m(f,g)\|_{p_3}\le C\|f\|_{p_1}\|g\|_{p_2}$$
for any $f, g\in{\mathcal P}(\R^n)$, i.e. $B_m$ extends to a
bounded bilinear operator from $L^{p_1}(\R^n)\times L^{p_2}(\R^n)
$ to $L^{p_3}(\R^n)$ (where we replace $L^\infty(\R^n)$ for
$C_0(\R^n)$ in the case $p_i=\infty$ for $i=1,2$).

We write $\MB_{(p_1,p_2,p_3)}(\R^n)$ for the space of bilinear
multipliers of type  $(p_1,p_2, p_3)$ and $\|m\|_{p_1,p_2,p_3}=
\|B_m\|$.
\end{defi}

The study of bilinear multipliers for smooth symbols (where
$m(\xi,\eta)$  is a ``nice" regular function) goes back to the
work by R.R. Coifman and Y. Meyer in \cite{CM}.

Particularly simple examples are  the following bilinear
convolution-type operators:
 For a given $K\in L_{loc}^1(\R^n)$ we  define
\begin{equation} \label{CK} C_K(f,g)(x)=\int_\R f(x-y)g(x+y) K(y)
dy \end{equation} for $f$ and  $g$ belonging to $C_{00}(\R^n)$.

If  $K\in L^1(\R^n)$ then $m(\xi,\eta)=\hat K(\xi-\eta)$ defines a
multiplier in $\MB_{(p_1,p_2,p_3)}(\R^n)$ for $1/p_1+1/p_2=1/p_3$
if $p_3\ge 1$ and $\|m\|_{p_1,p_2,p_3}\le \|K\|_1$.

Indeed, for $f$ and $g\in {\mathcal S}(\R)$, one has
$f(x-y)=\int_{\R^n} \hat f(\xi) e^{2\pi i\langle
x-y,\xi\rangle}d\xi$ and $g(x+y)=\int_{\R^n} \hat g(\eta) e^{2\pi
i\langle x+y,\eta\rangle}d\eta$. Hence we have \ba
C_K(f,g)(x)&=&\int_{\R^n} f(x-y)g(x+y) K(y) dy\\
&=&\int_{\R^n} \int_{\R^n} \int_{\R^n}\hat f(\xi)\hat g(\eta) K(y)
e^{2\pi i \langle x-y,\xi\rangle}e^{2\pi i\langle x+y,\eta\rangle}d\xi d\eta dy  \\
&=&\int_{\R^n} \int_{\R^n} \hat f(\xi)\hat g(\eta) (\int_{\R^n}
K(y) e^{-2\pi i\langle \xi-\eta,y\rangle} dy)e^{2\pi i\langle
\xi+\eta, x\rangle}d\xi d\eta   \\
&=&\int_{\R^n}\int_{\R^n} \hat g(\eta)\hat f(\xi)\hat K(\xi-\eta)
e^{2\pi i\langle\xi+\eta,x\rangle}d\xi d\eta.   \ea

  This motivates the introduction of the following class of
  multipliers.

\begin{defi} Let $1\le p_1, p_2\le\infty$ and $0< p_3\le \infty$. We
 denote by $\tilde\M_{(p_1,p_2,p_3)}(\R^n)$ the space of
  measurable functions $M:\R^n\to \C$  such that
 $m(\xi,\eta)=M(\xi-\eta)\in \MBR,$ that is to say
 $$ B_M(f,g)(x)=\int_{\R^n} \int_{\R^n}\hat f(\xi)\hat g(\eta)
 M(\xi-\eta)e^{2\pi i \langle \xi+\eta,x\rangle}d\xi d\eta$$
 extends to a bounded bilinear map from
 $L^{p_1}(\R^n)\times L^{p_2}(\R^n)$ into $L^{p_3}(\R^n)$.
 We keep the notation $\|M\|_{p_1,p_2,p_3}= \|B_M\|.$
\end{defi}

 It was
only in the last decade that the cases
$M_0(x)=\frac{1}{|x|^{1-\alpha}}$  were shown to define  bilinear
multipliers of type $(p_1,p_2, p_3)$ for
$1/p_3=1/p_1+1/p_2-\alpha$ for $1<p_1,p_2<\infty$ and
$0<\alpha<1/p_1+1/p_2$ (see (\ref{bfi}) in Theorem \ref{mt0}) and,
in the case $n=1$, $M_1(x)=-i sign (x)$ was shown to define a
bilinear multiplier of type $(p_1,p_2, p_3)$ for
$1/p_3=1/p_1+1/p_2$ for $1<p_1,p_2<\infty$ and $p_3>2/3$ (see
(\ref{bht}) in Theorem \ref{mt0}).
 These two main examples correspond to the following bilinear operators: the {\it
bilinear fractional integral} defined by {$$
I_\alpha(f,g)(x)=\int_{\R} \frac{f(x-y)g(x+y)}{|y|^{1-\alpha}}dy,
\quad 0<\alpha<1$$} and the {\it bilinear Hilbert transform}
defined by $$ H(f,g)(x)=\lim_{\e\to 0}\frac{1}{\pi}
\int_{|y|>\e}\frac{f(x-y)g(x+y)}{y}dy$$
 respectively.

Let us collect the results about their boundedness which are known
nowadays.
\begin{teor} \label{mt0}  \ Let $1<p_1,p_2<\infty$,
$0<\alpha<1/p_1+1/p_2$, $1/q=1/p_1+1/p_2-\alpha$,
$1/p_3=1/p_1+1/p_2$ and {$2/3<p_3<\infty$}. Then there exist
constants $A$ and $ B$ such that
\begin{equation}\label{bht}\|H(f,g)\|_{p_3}\le A\|f\|_{p_1}\|g\|_{p_2} \hbox{(Lacey-Thiele,  \cite{LT1, LT2}),}\end{equation}
\begin{equation}\label{bfi}
\|I_\alpha(f,g)\|_{q}\le B\|f\|_{p_1}\|g\|_{p_2}.\hbox{
(Kenig-Stein {\cite{KS}}, Grafakos-Kalton \cite{GK}
).}\end{equation}
\end{teor}

Our objective is to study the basic properties of the classes
$\mathcal{BM}_{(p_1,p_2,p_3)}(\R)$ and $\M_{p_1,p_2,p_3}(\R)$, to
find examples of bilinear multipliers in these classes and to get
methods to produce new ones.

As usual, if $f\in L^1(\R^n)$ we denote by $\tau_x$, $M_x$ and
$D^p_t$ the translation  $\tau_x f(y)=f(y-x)$ for $x\in \R^n$, the
modulation  $M_x f(y)= e^{2\pi i\langle x,y\rangle}f(y)$ and the
dilation $D^p_tf(x)=t^{-n/p}f(\frac{x}{t})$  for  $0<p,t<\infty$.

With this notation out of the way one has, for  $1\le p\le \infty$
and $1/p+1/p'=1,$
\begin{equation}\label{basic}\widehat{(\tau_x f)}(\xi)= M_{-x}\hat f(\xi),\quad \widehat {(M_x f)}(\xi)= \tau_x \hat
f(\xi), \quad \widehat{(D^p_t f)}(\xi)= D^{p'}_{t^{-1}}\hat
f(\xi).\end{equation}

Clearly $\tau_x, M_x$ and $D^p_t$ are isometries on $L^p(\R^n)$
for any $0< p\le\infty$.

Although most of the results presented in what follows have a
formulation in $n\ge 1$ we shall restrict ourselves to the case
$n=1$ for simplicity. The reader is referred to  \cite{B, BBCG,
BCG,BV, FS} for several similar results on other groups, and to
find same methods of transference.

\section{Bilinear multipliers: The basics}

Let us start by pointing out a characterization, for $p_3\ge 1$,
in terms of the duality, whose elementary proof is left to the
reader.

\begin{prop} \label{dual} Let $1\le p_3\le \infty$. Then
$m\in \MBR$ if and only if there exists $C>0$ such that
$$|\int_{\R^2} \hat f(\xi) \hat g(\eta) \hat h(\xi+\eta) m(\xi,\eta)d\xi
d\eta |\le C \|f\|_{p_1}\|g\|_{p_2}\|h\|_{p'_3}$$ for all $f, g,
h\in {\mathcal P}(\R)$.
\end{prop}

We now present a basic example of a bilinear multiplier. For    a
Borel regular measure in $\R$ $\mu$  we denote
$\hat\mu(\xi)=\int_\R e^{-2\pi i x\xi}d\mu(x)$ its Fourier
transform.

\begin{prop} \label{meas} Let $p_3\ge 1$ and
$1/p_1+1/p_2=1/p_3$ and let $m(\xi,\eta)=\hat\mu(\alpha\xi
+\beta\eta)$ where $\mu$ is a Borel regular measure in $\R$ and
$(\alpha,\beta)\in \R^2$. Then $m\in \MBR$ and
$\|m\|_{p_1,p_2,p_3}\le \|\mu\|_1$.
\end{prop}

\begin{proof} Let us first rewrite the value $B_m (f,g)$ as
follows: \ba B_m(f,g)(x)&=&\int_{\R^2} \hat f(\xi)\hat g(\eta)
\hat \mu(\alpha\xi+\beta\eta)e^{2\pi i (\xi+\eta)x}d\xi d\eta\\
&=&\int_{\R^2} \hat f(\xi)\hat g(\eta) (\int_\R e^{-2\pi
i(\alpha\xi+
\beta\eta)t} d\mu(t))e^{2\pi i (\xi+\eta)x}d\xi d\eta\\
&=&\int_\R (\int_{\R^2} \hat f(\xi)\hat g(\eta)e^{2\pi i (x-\alpha
t)\xi}e^{2\pi i(x-\beta t)\eta}d\xi d\eta )d\mu(t)\\&=&\int_\R
f(x-\alpha t)g(x-\beta t)d\mu(t).\ea Hence, using Minkowski's
inequality, one has
\begin{eqnarray*} \|B_m(f,g)\|_{p_3}&\le& \int_\R \|f(\cdot-\alpha t)g(\cdot-\beta t)\|_{p_3}d|\mu|(t)\\
&\le& \int_\R \|f(\cdot-\alpha t)\|_{p_1}\|g(\cdot-\beta t)\|_{p_2}d|\mu|(t)\\
&=& \|f\|_{p_1}\|g\|_{p_2}\int_\R
d|\mu|(t)=\|\mu\|_1\|f\|_{p_1}\|g\|_{p_2}.
 \end{eqnarray*}
\end{proof}

Let us start with some elementary properties  of the bilinear
multipliers when composing with translations, modulations and
dilations.

\begin{prop}\label{p1} Let $m\in \MB_{(p_1,p_2,p_3)}(\R)$.

\begin{enumerate}[\sf \ \ \ (a)]
\item If $m_1\in \tilde\M_{s_1,p_1}(\R)$ and $m_2\in \tilde\M_{s_2,p_2}(\R)$
then $m_1(\xi)m(\xi,\eta)m_2(\eta)\in \MB_{(s_1,s_2,p_3)}(\R)$.
Moreover $$ \|m_1 m m_2\|_{s_1,s_2, p_3}\le \|m_1 \|_{s_1,p_1}\| m
\|_{p_1,p_2, p_3}\| m_2\|_{s_2,p_2}$$

\item  $\tau_{(\xi_0,\eta_0)} m\in \MBR$ for each $(\xi_0,\eta_0)\in
\R^2$ and $$ \|\tau_{(\xi_0,\eta_0)} m\|_{p_1,p_2, p_3}= \|
m\|_{p_1,p_2, p_3}.$$

\item $M_{(\xi_0,\eta_0)} m\in \MBR$ for each $(\xi_0,\eta_0)\in
\R^2$ and $$ \|M_{(\xi_0,\eta_0)} m\|_{p_1,p_2, p_3}= \|
m\|_{p_1,p_2, p_3}$$

\item If $\frac{2}{q}=\frac{1}{p_1}+\frac{1}{p_2}- \frac{1}{p_3}$ and $0<t<\infty$ then  $D_t^q m\in \MBR$ and $$\|D_t^q m \|_{p_1,p_2,
p_3}= \| m\|_{p_1,p_2, p_3}.$$

\end{enumerate}
\end{prop}
\begin{proof}
Use (\ref{basic}) to deduce the following formulas   \be B_{m_1 m
m_2}(f,g)= B_m(T_{m_1}f, T_{m_2}g).\ee \be
B_{\tau_{(\xi_0,\eta_0)} m}(f,g)= M_{\xi_0+\eta_0}B_m(
M_{-\xi_0}f,M_{-\eta_0}g).\ee \be B_{M_{(\xi_0,\eta_0)} m}(f,g)=
B_m( \tau_{-\xi_0}f,\tau_{-\eta_0}g).\ee \be \label{dilatacion}
B_m(D_t^{p_1}f,D_t^{p_2}g)=D_t^{p_3}B_{D^q_t m}(f,g).\ee Let us
check only the validity of last one. The other ones follow easily
from the previous facts.
 \ba
 B_m(D_t^{p_1}f,D_t^{p_2}g)(x)&=& \int_{\R^2}t^{\frac{1}{p_1'}}\hat
 f(t\xi)t^{\frac{1}{p_2'}}\hat g(t\eta)m(\xi, \eta) e^{2\pi
 i(\xi+\eta)x}d\xi d\eta\\
 &=& \int_{\R^2}t^{\frac{1}{p_1'}}\hat
 f(\xi)t^{\frac{1}{p_2'}}\hat g(\eta)m(\frac{\xi}{t}, \frac{\eta}{t}) e^{2\pi
 i(\xi+\eta)\frac{x}{t}}t^{-2}d\xi d\eta\\
 &=& t^{-\frac{1}{p_3}}\int_{\R^2}\hat
 f(\xi)\hat g(\eta)t^{-\frac{1}{p_1}-\frac{1}{p_2}+ \frac{1}{p_3}} m(\frac{\xi}{t}, \frac{\eta}{t}) e^{2\pi
 i(\xi+\eta)\frac{x}{t}}d\xi d\eta\\
 &=& D_t^{p_3}B_{D^q_t m}(f,g)(x).
 \ea

\end{proof}

From (\ref{dilatacion}) we can see that the condition
$1/p_1+1/p_2=1/p_3$ is also connected to the homogeneity of the
symbol.

\begin{prop} Let $m\in \MBR$ such that $m(t\xi,t\eta)=m(\xi,\eta)$ for any $t>0$. Then $\frac{1}{p_1}+\frac{1}{p_2}=\frac{1}{p_3}$.
\end{prop}
\begin{proof} From assumption $D_t^\infty m= m$. Using (\ref{dilatacion})
 we have $$B_m(D_t^{p_1}f, D_t^{p_2}g)=
t^{1/p_3-(1/p_1+1/p_2)}D^{p_3}_tB_{m}(f, g)$$ and therefore
\ba\|B_m(f,g)\|_{p_3}&=&\|D_t^{p_3}B_m(f,g)\|_{p_3}\\
&=&t^{-1/p_3+(1/p_1+1/p_2)}\|B_m(D^{p_1}f,D_t^{p_2}g)\|_{p_3}\\
&\le &t^{-1/p_3+(1/p_1+1/p_2)}\|B_m\|\|f\|_{p_1}\|g\|_{p_2}.
 \ea
For this to hold for any $0<t<\infty$ one needs
$1/p_1+1/p_2=1/p_3$.\end{proof}

Let us combine the previous results to get new bilinear
multipliers from a given one.
\begin{prop} \label{convo} Let $p_3\ge 1$  and $m\in \MBR$.

\begin{enumerate}[\sf \ \ \ (a)]

\item If $Q=[a,b]\times [c,d]$ and $1<p_1,p_2<\infty$ then $m\chi_{Q}\in \MBR$ and
 $\|m\chi_{Q}\|_{p_1,p_2,p_3}\le C \|m\|_{p_1,p_2,p_3}.$

\item If $\Phi\in L^1(\R^2)$ then $\Phi*m\in \MBR$ and
 $\|\Phi*m\|_{p_1,p_2,p_3}\le \|\Phi\|_1 \|m\|_{p_1,p_2,p_3}.$

\item If $\Phi\in L^1(\R^2)$  then $\hat\Phi m\in \MBR$ and
 $\|\hat\Phi m\|_{p_1,p_2,p_3}\le \|\Phi\|_1
\|m\|_{p_1,p_2,p_3}.$

\item If $\psi\in L^1(\R^+, t^{\frac{1}{p_3}-(\frac{1}{p_1}+\frac{1}{p_2}})$  then $m_\psi(\xi,\eta)=\int_0^\infty m(t\xi,t\eta)\psi(t)dt \in \MBR$.
Moreover $\|m_\psi\|_{p_1,p_2,p_3}\le \|\psi\|_1
\|m\|_{p_1,p_2,p_3}.$

\end{enumerate}
\end{prop}

\begin{proof} (a) Use that $\chi_{[a,b]}\in {\tilde{\mathcal M}}_{p_1,p_1}$ for $1<p_1<\infty$ and $\chi_{[c,d]}\in {\tilde{\mathcal M}}_{p_2,p_2}$ for
$1<p_2<\infty$ together with Proposition \ref{p1} part (a).

(b) Note that \ba B_{\Phi* m}(f,g)(x)&=&\int_{\R^2} \hat f(\xi)
\hat g(\eta) (\int_{\R^2} m(\xi-u, \eta-v) \Phi(u,v)dudv)e^{2\pi
i(\xi+\eta)x}d\xi d\eta\\
&=&\int_{\R^2} \big(\int_{\R^2}\hat f(\xi) \hat g(\eta)
 m(\xi-u, \eta-v) e^{2\pi
i(\xi+\eta)x}d\xi d\eta\big)\Phi(u,v)dudv\\
&=&\int_{\R^2} B_{\tau_{(u,v)}m}(f, g)(x)\Phi(u,v)dudv. \ea

 From the vector-valued Minkowski
inequality  and Proposition \ref{p1} part (b), we have \ba
\|B_{\Phi* m}(f,g)\|_{p_3}&\le &\int_{\R^2}
\|B_{\tau_{(u,v)}m}(f, g)\|_{p_3} |\Phi(u,v)|dudv\\
&\le& \|m\|_{p_1,p_2,p_3} \|f\|_{p_1}\|g\|_{p_2} \|\Phi\|_1.\ea

(c) Observe that\ba B_{\hat \Phi  m}(f,g)(x)&=&\int_{\R^2} \hat
f(\xi) \hat g(\eta) (\int_{\R^2} M_{(-u, -v)}m(\xi, \eta)
\Phi(u,v)dudv)e^{2\pi
i(\xi+\eta)x}d\xi d\eta\\
&=&\int_{\R^2}  B_{M_{(-u, -v)}m}(f,g)(x)\Phi(u,v)dudv.\\
\ea

Argue as above, using now Proposition \ref{p1} part (c), to
conclude the result.

(d) Use now Proposition \ref{p1} part (d), for
$\frac{1}{p_3}-(\frac{1}{p_1}+\frac{1}{p_2})= -\frac{2}{q}$,   \ba
B_{m_\psi}(f,g)(x)&=&\int_{\R^2} \hat f(\xi) \hat g(\eta)
(\int_{0}^\infty D^q_{t^{-1}}m(\xi, \eta) t^{-2/q} \psi(t)
dt)e^{2\pi
i(\xi+\eta)x}d\xi d\eta\\
&=&\int_{0}^\infty  B_{D^q_{t^{-1}}m}(f,g)(x)t^{-2/q}
\psi(t) dt.\\
\ea
\end{proof}

With all these procedures we have several useful methods to
produce multipliers in $\MBR$. Let us mention one application of
each of them.

\begin{ejem}
\begin{enumerate}[\sf \ \ \ (1)]
\item
If $\frac{1}{p_1}+\frac{1}{p_2}=\frac{1}{p_3}$, $m_1\in
\tilde\M_{(p_1,p_1)}$ and $m_2\in \tilde\M_{(p_2,p_2)}$ then
$m(\xi,\eta)=m_1(\xi)m_2(\eta) \in \MB_{p_1,p_2, p_3}.$

\item If $m\in \MBR$, $p_3\ge 1$ and $Q_1, Q_2$ are bounded measurable sets in $\R$ then $$\frac{1}{|Q_1||Q_2|}\int_{Q_1\times
Q_2} m(\xi+u, \eta+v) du dv\in \MBR.$$

\item If $\Phi \in L^1(\R^2)$ then $\hat \Phi\in \MBR$ for
$\frac{1}{p_1}+\frac{1}{p_2}=\frac{1}{p_3}$, $p_3\ge 1$.

\item  If $m\in \MBR$, $|\frac{1}{p_1}+\frac{1}{p_2}- \frac{1}{p_3}|<1$
then $$m_1(\xi,\eta)=\int_0^\infty m(t\xi,t\eta)\frac{dt}{1+t^2}
\in \MBR.$$
\end{enumerate}
\end{ejem}

A combination of the previous results gives the following examples
of bilinear multipliers in ${\mathcal BM}_{(1,1,p_3)}(\R)$ whose
proof is left to the reader.
\begin{coro} Let $\Phi\in L^1(\R^2)$, $\psi_1\in L^{p_1}(\R)$ and
$\psi_2\in L^{p_2}(\R)$ and
$\frac{1}{p_1}+\frac{1}{p_2}=\frac{1}{p_3}\le 1$ then
$$m(\xi,\eta)=\hat\psi_1(\xi)
\hat\Phi(\xi,\eta)\hat\psi_2(\eta)\in {\mathcal
BM}_{(1,1,p_3)}(\R).$$
\end{coro}

Let us use Proposition \ref{dual}  and interpolation to get a
sufficient integrability condition to guarantee that $m\in \MBR$.

\begin{teor}\label{interp}

Let  $1\le p_1,p_2\le p\le 2$ and $p_3\ge p'$ such that
$\frac{1}{p_1}+\frac{1}{p_2}-\frac{2}{p}=\frac{1}{p_3}$. If $m\in
L^{p}(\R^2)$ then $m\in {\mathcal {BM}}_{(p_1,p_2,p_3)}(\R)$.

\end{teor}

\begin{proof}
Let us show first that $m\in {\mathcal {BM}}_{(p,p,\infty)}(\R).$
Let $f, g\in L^p(\R)$ and $h\in L^{1}(\R)$. Using H\"older and
Hausdorff-Young's inequalities one gets \ba|\int_{\R^2} \hat
f(\xi) \hat g(\eta) \hat h(\xi+\eta) m(\xi,\eta)d\xi d\eta |&\le&
\|m\|_{L^p(\R^2)} \|\hat h\|_\infty\|\hat f\|_{p'}\|\hat
g\|_{p'}\\
&\le& \|m\|_{L^p(\R^2)}\|h\|_1\|f\|_p \|g\|_p.\ea

Similarly, changing the variables $\xi+\eta=u$, $\xi=-v$, one has
$$\int_{\R^2} \hat f(\xi) \hat
g(\eta) \hat h(\xi+\eta) m(\xi,\eta)d\xi d\eta = \int_{\R^2} \hat
f(-v) \hat g(u+v) \hat h(u) m(-v,u+v)dv du.$$ An argument as above
gives also the estimate
$$|\int_{\R^2} \hat
f(-v) \hat g(u+v) \hat h(u) m(-v,u+v)dv du |\le
\|m\|_{L^p(\R^2)}\|g\|_1\|f\|_p \|h\|_p.$$ This shows that $m\in
{\mathcal {BM}}_{(p,1,p')}(\R)$. A similar argument shows also
that $m\in {\mathcal {BM}}_{(1,p,p')}(\R)$.

 Given $1\le
\tilde p_1\le p$ and $p'\le \tilde p_3\le \infty$ with
$\frac{1}{\tilde p_1}-\frac{1}{\tilde p_3}=\frac{1}{p}$ we have
$0\le\theta\le 1$ such that $\frac{1}{\tilde
p_1}=\frac{1-\theta}{p}+ \frac{\theta}{1}$ and $\frac{1}{\tilde
p_3}=\frac{1-\theta}{\infty}+ \frac{\theta}{p'}$. Hence, by
interpolation, $m\in {\mathcal {BM}}_{(\tilde p_1,p,\tilde
p_3)}(\R)$.

Similarly  $m\in {\mathcal {BM}}_{(p,\tilde p_2,\tilde q_3)}(\R)$
whenever $1\le \tilde p_2\le p$ and $p'\le \tilde q_3\le \infty$
with $\frac{1}{\tilde p_2}-\frac{1}{\tilde q_3}=\frac{1}{p}$.

To finish the proof we observe that if $1< p_1<p$ and $1<p_2<p$
then for each $0<\theta<1$ there exist $ 1\le \tilde p_1\le p_1<p$
and $ 1\le \tilde p_2\le p_2<p$ such that $$\frac{1}{p_1}-
\frac{1}{p}=(1-\theta)(\frac{1}{\tilde p_1}- \frac{1}{p}),\quad
\frac{1}{p_2}- \frac{1}{p}=\theta(\frac{1}{\tilde p_2}-
\frac{1}{p}).$$ Denoting $\tilde p_3, \tilde q_3$ the values such
that $\frac{1}{\tilde p_2}- \frac{1}{p}=\frac{1}{\tilde p_3}$ and
$ \frac{1}{\tilde p_2}- \frac{1}{p}=\frac{1}{\tilde q_3}$ one
obtains that
$$\frac{1}{p_1}=
\frac{(1-\theta)}{\tilde p_1}+ \frac{\theta}{p}, \quad
\frac{1}{p_2}= \frac{(1-\theta)}{p}+ \frac{\theta}{\tilde p_1},
\quad \frac{1}{p_3}= \frac{(1-\theta)}{\tilde p_3}+
\frac{\theta}{\tilde q_3}.$$ Hence the  result follows again from
interpolation between the last ones.
\end{proof}

\section{Bilinear multipliers defined by functions in one variable}

Let us restrict ourselves to a smaller family of multipliers where
$m(\xi, \eta)=M(\xi-\eta)$ for some $M$ defined in $\R$.  These
multipliers satisfy \begin{equation} B_m(M_xf, M_xg) =
M_{2x}B_m(f,g) .\end{equation}As in the introduction we use the
notation $\tilde\M_{p_1,p_2,p_3}(\R)$ for the space of
 functions $M:\R\to \C$  such that
 $m(\xi,\eta)=M(\xi-\eta)\in \MBR,$ that is to say
 $$ B_M(f,g)(x)=\int_{\R^2} \hat f(\xi)\hat g(\eta)
 M(\xi-\eta)e^{2\pi i (\xi+\eta)x}d\xi d\eta,$$
defined for $\hat f$ and $\hat g$ compactly supported, extends to
a bounded bilinear map from
 $L^{p_1}(\R)\times L^{p_2}(\R)$ into $L^{p_3}(\R)$.
 We keep the notation $\|M\|_{p_1,p_2,p_3}= \|B_M\|.$

The reader should be aware that the starting assumption on the
function $M$ is only relevant for the definition of  the bilinear
mapping to make sense when acting on certain classes of ``nice"
functions. Then a density argument allows to extend functions
belonging to Lebesgue spaces. We would like to point out the
following observation.

\begin{nota} If $M_n\in \mul$ are functions such that $M_n(x)\to
M(x)$ a.e and   $\sup_{n}\|M_n\|<\infty$ then $M\in \mul$ and
$\|M\|_{p_1,p_2,p_3}\le\sup_{n}\|M_n\|_{p_1,p_2,p_3}$.

Indeed, this  fact follows from Fatou's lemma, since
$$
\|B_M(f,g)\|_{p_3}\le \liminf \|B_{M_n}(f,g)\|_{p_3}\le
\sup_{n}\|M_n\|_{p_1,p_2,p_3}\|f\|_{p_1}\|g\|_{p_2}.$$
\end{nota}

\begin{nota} The case $M(x)=\frac{1}{|x|^{1-\alpha}}$ (and even the
$n$-dimensional case) corresponds to the bilinear fractional
integral. This was first shown by C. Kenig and E. Stein in
\cite{KS} to belong to $\mul$ for any $1<p_1,p_2<\infty$,
$0<\alpha<1/p_1+1/p_2$ and $1/p_1+1/p_2=1/p_3-\alpha$.
 Another  very important and non trivial example is the bilinear Hilbert
transform, given by $M(x)=-i sign(x)$, which was shown by M. Lacey
and C.Thiele in \cite{LT1, LT2, LT3} to belong to $\mul$ for any
$1<p_1,p_2<\infty$, $1/p_1+1/p_2=1/p_3$ and $p_3> 2/3$. These
results were extended to other cases  in \cite{GK} and \cite{GN1,
GN2} respectively.
\end{nota}
We start reformulating  the definition of this class of bilinear
multipliers.

\begin{prop} Let $M\in L^1_{loc}(\R)$, $f, g\in {\mathcal P}(\R)$.
Then \be \label{exp1} B_M(f,g)(x)=\frac{1}{2}\int_{\R^2} \hat
f(\frac{u+v}{2})\hat g(\frac{u-v}{2})
 M(v)e^{2\pi i u x}du dv\ee

 \begin{equation}\label{f1}B_M(f,g)(-x)= \int_{\R}
(\widehat{\tau_x g}*M)(\xi) \widehat{\tau_x
f}(\xi)d\xi.\end{equation}

\begin{equation}\label{f2}\widehat{B_M(f,g)}(x)= \frac{1}{2}C_{M}(\widehat{D^1_{1/2} f}, \widehat{D^1_{1/2} g}) (x).\end{equation}
\end{prop}
\begin{proof} (\ref{exp1})  follows changing  variables.

 To show (\ref{f1}) observe that
\ba B_M(f,g)(-x)&=&\int_{\R^2} \widehat {\tau_xf}(\xi)\widehat
{\tau_x g}(\eta)
 M(\xi-\eta)d\xi d\eta \\
 &=&\int_{\R} (\int_\R\widehat
{\tau_x g}(\eta)
 M(\xi-\eta)d\eta) \widehat {\tau_xf}(\xi)d\xi\\
 &=&\int_{\R}
(\widehat{\tau_x g}*M)(\xi) \widehat{\tau_x f}(\xi)d\xi
 \ea

 Finally, using (\ref{exp1}), we have \ba B_M(f,g)(x)
&=& \frac{1}{2}\int_{\R}\Big(\int_\R\hat f(\frac{u+v}{2}) \hat
g(\frac{u-v}{2})
M(v) dv\Big) e^{2\pi i ux}dv\\
&=& \frac{1}{2}\int_{\R} C_{M}(D^\infty_{1/2}\hat f,
D^\infty_{1/2}\hat g)(u)
 e^{2\pi i ux}du .\\
\ea  This implies (\ref{f2}).
\end{proof}

 For symbols $M$ which are integrable we can write  $B_M$ in terms
 $C_K$ for some  kernel $K$.

\begin{prop}  \label{igual} Let  $M\in  L^1(\R)$ and set $K(t)=\hat M(-t)$. Then $B_M=
C_K$, i.e
$$B_M(f,g)= \int_\R f(x-t)g(x+t)K(t)dt$$
\end{prop}

\begin{proof}  \ba C_K(f,g)(x)&=&\int_\R f(x-t)g(x+t)\hat
M(-t)dt\\
&=&\int_\R (\int_{\R^2} \hat f(\xi)\hat g(\eta)e^{2\pi i
(x-t)\xi}e^{2\pi i(x+t)\eta}d\xi d\eta )\hat M(-t)dt\\
&=&\int_{\R^2} \hat f(\xi)\hat g(\eta) (\int_\R \hat M(t)e^{2\pi i
(\xi-\eta)t}dt) e^{2\pi i (\xi+\eta)x}d\xi d\eta \\
&=& B_M(f,g)(x).\ea
\end{proof}

This class does have much richer properties than $\MBR$. As above
use the notation $f_t(x)=D^1_tf(x)=\frac{1}{t}f(\frac{x}{t})$ for
a function $f$ defined in $\R$. The following facts are immediate.

\be \label{tras} \tau_y B_M(f,g)= B_M(\tau_y f, \tau_y g) , y\in
\R.\ee

\be \label{mod} M_{2y} B_M(f,g)= B_M(M_y f, M_y g) , y\in \R. \ee

\be \label{dil} (B_M(f,g))_t= B_{D^1_{t^{-1}}M}( f_t,  g_t) ,
t>0.\ee

When specializing the properties obtained for $m(\xi,\eta)$ to the
case $M(\xi-\eta)$ we get the following facts:

\be \label{tras1} B_M(\tau_{-y} f, \tau _{y} g)= B_{M_y M}(f,  g)
, y\in \R.\ee

\be \label{mod1}  B_M(M_y f,M_{-y} g)= B_{\tau_{2y}M}(f, g) , y\in
\R. \ee

For $ \frac{1}{q}= \frac{1}{p_1}+\frac{1}{p_2}-\frac{1}{p_3}$ we
have \be \label{dil2} B_M(D^{p_1}_tf,D^{p_2}_tg)=
D^{p_3}_tB_{D^q_{t}M}( f, g) , t>0.\ee

As in the previous section we can  generate  new multipliers in
$\mul$.

\begin{prop}  \label{propi2} Let $p_3\ge 1$, $\phi\in L^1(\R)$ and $M\in
\mul$. Then

\begin{enumerate}[\sf \ \ \ (a)]
\item
$\phi*M\in \mul$ and $\|\phi*M\|_{p_1,p_2,p_3}\le \|\phi\|_1
\|M\|_{p_1,p_2,p_3}.$
\item   $\hat\phi M\in \mul$ and
 $\|\hat\phi M\|_{p_1,p_2,p_3}\le \|\phi\|_1
\|M\|_{p_1,p_2,p_3}.$

\item If $\psi\in L^1(\R^+, t^{\frac{1}{p_3}-(\frac{1}{p_1}+\frac{1}{p_2})})$  then $M_\psi(\xi)=\int_0^\infty M(t\xi)\psi(t)dt \in \mul$.
Moreover $\|M_\psi\|_{p_1,p_2,p_3}\le \|\psi\|_1
\|M\|_{p_1,p_2,p_3}.$
\end{enumerate}

\end{prop}

\begin{proof}
(a) Apply  Minkowski's inequality to the following fact: \ba
B_{\phi*M}(f,g)(x)&=& \int_{\R^2} \hat f(\xi)\hat g(\eta)
 (\int_\R M(\xi-\eta-u)\phi(u)du)e^{2\pi i (\xi+\eta)x}d\xi d\eta\\
 &=& \int_\R (\int_{\R^2} \widehat {M_{-u}f}(\xi )\hat g(\eta)
  M(\xi-\eta)e^{2\pi i (\xi+\eta)x}d\xi d\eta)e^{2\pi i u
  x}\phi(u)du\\
  &=& \int_\R M_u B_M( M_{-u}f,g)(x)\phi(u)du.
 \ea
 (b) Observe that\ba B_{\hat \phi  m}(f,g)(x)&=&\int_{\R^2} \hat
f(\xi) \hat g(\eta) (\int_{\R} (M_{-u}m)(\xi-\eta)
\phi(u)du)e^{2\pi
i(\xi+\eta)x}d\xi d\eta\\
&=&\int_{\R^2}  B_{M_{-u}m}(f,g)(x)\phi(u)du.\\
\ea

Use now Minkowski's again and (\ref{tras1}).

(c) Write $\frac{1}{p_3}-(\frac{1}{p_1}+\frac{1}{p_2})=
-\frac{1}{q}$,   \ba B_{M_\psi}(f,g)(x)&=&\int_{\R^2} \hat f(\xi)
\hat g(\eta) (\int_{0}^\infty D^q_{t^{-1}}M(\xi) t^{-1/q} \psi(t)
dt)e^{2\pi
i(\xi+\eta)x}d\xi d\eta\\
&=&\int_{0}^\infty  B_{D^q_{t^{-1}}M}(f,g)(x)t^{-1/q} \psi(t) dt.
\ea The result follows from (\ref{dil2}) and Minkowski's again.
\end{proof}

\begin{prop}  \label{mix} Let $p_3\ge 1$, $\phi\in L^1(\R)$ and $M\in
\mul$. Then $m(\xi,\eta)= M(\xi-\eta)\hat \phi(\xi+\eta)\in \MBR$
and $\|m\|_{p_1,p_2,p_3}\le \|\phi\|_1 \|M\|_{p_1,p_2,p_3}.$
\end{prop}
\begin{proof} Apply Young's inequality to the following fact: \ba
B_{m}(f,g)(x)&=& \int_{\R^2} \hat f(\xi)\hat g(\eta)
 M(\xi-\eta)(\int_\R \phi(y)e^{-2\pi i (\xi+ \eta)y} dy)e^{2\pi i (\xi+\eta)x}d\xi d\eta\\
 &=& \int_\R (\int_{\R^2} \hat f(\xi )\hat g(\eta)
  M(\xi-\eta)e^{2\pi i (\xi+\eta)(x-y)}d\xi d\eta)\phi(y)dy\\
  &=& \phi* B_M(f,g)(x).
 \ea
\end{proof}

Let  us show that the classes $\mul$ are reduced to $\{0\}$ for
some values of the parameters.

\begin{teor} Let $p_3\ge 1$ such that
$\frac{1}{p_1}+\frac{1}{p_2}<\frac{1}{p_3}$. Then $\mul=\{0\}$.
\end{teor}

\begin{proof} Let $M\in \mul$. Using Proposition \ref{propi2} we have that $\phi*M \in \mul$ for any $\phi$ continuous with compact support.
Hence we may assume that $M\in  L^1(\R)$. Using Proposition
\ref{igual} one has that
$$B_M(f,g)(x)=\int_{(x+B_R)\cap
(-x+B_R)}f(x-t)g(x+t)\hat M(-t)dt$$ for any $f$ and $g$
continuous functions supported in a ball $B_R=\{|x|\le R\}$.
Therefore one concludes that $supp (B_M(f,g))\subset B_{2R}$ in
such a case. On the other hand  for any compactly supported
 function  $h$, $0<p< \infty$ and $y$ big enough one can say that $\| h \pm \tau_y
f\|_{p}= 2^{1/p}\|f\|_p$.

Consider $\{r_k\}$ the Rademacher system in $[0,1]$ and observe
that, for each $N\in \N$ and $y\in \R$, the orthonormality of the
system gives
$$ \int_0^1B_M(\sum_{k=0}^N r_k(t) \tau_{ky }f, \sum_{k=0}^N r_k(t) \tau_{ky }f) dt= \sum_{k=0}^N B_M(\tau_{ky }f,\tau_{ky }g)$$
Therefore, since $ \sum_{k=0}^N B_M(\tau_{ky }f,\tau_{ky
}g)=\sum_{k=0}^N \tau_{ky }B_M(f,g)$,  we conclude that  for $y$
big enough
$$ \|\sum_{k=0}^N \tau_{ky }B_M(f,g)\|_{p_3}^{p_3}= (N+1) \|B_M(f,g)\|^{p_3}_{p_3}.$$

On the other hand, for $p_3\ge 1$,
 \ba &&\|\int_0^1B_M(\sum_{k=0}^N
r_k(t) \tau_{ky }f,\sum_{k=0}^N r_k(t) \tau_{ky }g)
dt\|_{p_3}\\
&\le&\int_0^1\|B_M(\sum_{k=0}^N r_k(t) \tau_{ky }f,
\sum_{k=0}^N r_k(t) \tau_{ky }g) \|_{p_3}dt\\
&\le& \int\|B_M\|\|\sum_{k=0}^N r_k(t) \tau_{ky }f\|_{p_1}
\|\sum_{k=0}^N r_k(t) \tau_{ky }g) \|_{p_2}dt\\
&\le& \|B_M\|\sup_{0<t<1}\|\sum_{k=0}^N r_k(t) \tau_{ky }f\|_{p_1}
\sup_{0<t<1}\|\sum_{k=0}^N r_k(t) \tau_{ky }g \|_{p_2}\\
&\le& \|B_M\|(N+1)^{1/p_1}\|f\|_{p_1}
(N+1)^{1/p_2}\|g \|_{p_2}.\\
\ea

This implies that $(N+1)^{1/p_3} \|B_M(f,g)\|^{p_3}\le C
(N+1)^{1/p_1+1/p_2}\|f\|_{p_1}\|g \|_{p_2}.$ Hence $1/p_1+1/p_2\ge
1/p_3$.
\end{proof}

The following elementary lemma is quite useful to get necessary
conditions on multipliers.

\begin{lema} \label{ml} Let $M\in \mul$. If $\frac{1}{q}=\frac{1}{p_1}+\frac{1}{p_2}-\frac{1}{p_3}$   then there exists $C>0$ such that
$$|\int_\R e^{-\lambda^2 \xi^2}M(\xi) d\xi|\le C
\|M\|_{p_1,p_2,p_3}\lambda^{\frac{1}{q}-1}$$ for any $\lambda>0$.
\end{lema}
\begin{proof} Let $\lambda>0$ and denote $G_\lambda$ such that
$\hat G_\lambda(\xi)=e^{-2\lambda^2 \xi^2}$. Using (\ref{exp1})
one concludes that \ba B_M(G_\lambda,
G_\lambda)(x)&=&\frac{1}{2}\int_{\R^2} e^{-\lambda^2v^2}
e^{-\lambda^2u^2} M(v)e^{2\pi iu x}du dv\\
 &=&\frac{1}{2}(\int_{\R} e^{-\lambda^2v^2} M(v) dv)
(\frac{1}{\lambda}\int_\R e^{-u^2} e^{2\pi iu\frac{x}{\lambda}}du) \\
&=&C \frac{1}{\lambda}e^{-\pi^2 \frac{x^2}{\lambda^2}}(\int_{\R}
e^{-\lambda^2v^2} M(v) dv).\ea

Since $\|G_\lambda\|_p= C_p \lambda^{\frac{1}{p} -1}$ and  $M\in
\mul$ one gets that
$$\|B_M(G_\lambda, G_\lambda)\|_{p_3}= C|\int_{\R} e^{-\lambda^2v^2} M(v)
dv|\lambda^{\frac{1}{p_3} -1}\le C\|M\|_{p_1,p_2,p_3}
\lambda^{\frac{1}{p_1} -1}\lambda^{\frac{1}{p_2} -1}.$$

Therefore $|\int_\R e^{-\lambda^2 \xi^2}M(\xi) d\xi|\le C
\|M\|_{p_1,p_2,p_3}\lambda^{\frac{1}{q}-1}$.
\end{proof}

\begin{teor} \label{ci} If there exists   a non-zero continuous and integrable function $M$ belonging to $\mul$ then
$$\frac{1}{p_3}\le \frac{1}{p_1}+\frac{1}{p_2}\le  \frac{1}{p_3}+1.$$
\end{teor}

\begin{proof} Assume first that  $\frac{1}{p_1}+\frac{1}{p_2}<\frac{1}{p_3}$.
Use Lemma \ref{ml} applied to $\tau_{-2y}M$ for any $y\in \R$
together with (\ref{tras1}) to obtain
$$|\lambda\int_\R e^{-\lambda^2 \xi^2}M(\xi+2y) d\xi|\le C
\|M\|_{p_1,p_2,p_3}\lambda^{\frac{1}{q}}.$$ Therefore, using the
continuity of $M$  and $q<0$ one gets $$\lim_{\lambda\to
\infty}|\lambda\int_\R e^{-\lambda^2 \xi^2}M(\xi+2y)
d\xi|=|M(2y)|= 0.$$ Hence $M=0$.

Assume now that $\frac{1}{p_1}+\frac{1}{p_2}-\frac{1}{p_3}>1$.
Using again Lemma \ref{ml}, applied to $M_{y}M$, together with
(\ref{mod1}) we obtain
$$|\int_\R e^{-\lambda^2 \xi^2}M(\xi) e^{2\pi i y\xi}d\xi|\le C
\|M\|_{p_1,p_2,p_3}\lambda^{\frac{1}{q}-1}.$$ Therefore, taking
limits again as $\lambda\to 0$, since  $1/q-1>0$ we get $|\hat
M(y)|=0$. Hence $M=0$.
\end{proof}

\begin{coro}(see \cite[Prop 3.1]{V2}) Let $ p_3\ge 1$ such that
$\frac{1}{p_1}+\frac{1}{p_2}<\frac{1}{p_3}$ or
$\frac{1}{p_1}+\frac{1}{p_2}>\frac{1}{p_3}+1$. Then $\mul=\{0\}$.
\end{coro}

\begin{proof} Let $M\in \mul$. From Proposition \ref{propi2} we have that $\phi*M \in \mul$ for any $\phi$ compactly supported and continuous. Now use
Theorem \ref{ci}
 to conclude that $\phi*M =0$ for any  compactly supported and continuous $\phi$.
This implies that $M=0$.
\end{proof}

Let us now use  some interpolation methods to get more examples of
multipliers in $\mul$. First note that, selecting $\alpha=1$ and
$\beta=-1$ in Proposition \ref{meas} we obtain the following
simple example.
\begin{prop}\label{m1} If $\mu\in M(\R)$ then $M=\hat \mu \in
\mul$ for $\frac{1}{p_1}+\frac{1}{p_2}=\frac{1}{p_3}\le 1$ and
$\|M\|\le \|\mu\|_1$.
\end{prop}

\begin{teor}  Let
$\frac{1}{p_3}\le\frac{1}{p_1}+\frac{1}{p_2}\le \min\{2,
\frac{1}{p_3}+1\}$. If $M\in L^1(\R)$ and $M=\hat K$ for some
$K\in L^q(\R)$ where
$\frac{1}{p_1}+\frac{1}{p_2}-\frac{1}{p_3}=1-\frac{1}{q}$ then
$M\in \mul$  with $\|M\|_{p_1,p_2,p_3}\le C\|K\|_q$.
\end{teor}

\begin{proof} Consider the trilinear form $$T(K, f,g)=\int_\R
f(x-t)g(x+t)K(t)dt.$$ From Proposition \ref{igual} we have
$B_M(f,g)= T(K,f,g)$ for $M=\hat K$. Now use Proposition \ref{m1}
to conclude  that $T$ is bounded in $L^1(\R)\times
L^{q_1}(\R)\times L^{q_2}(\R) \to L^{s_1}(\R)$ where
$\frac{1}{q_1}+\frac{1}{q_2}=\frac{1}{s_1}\le 1$ and it has norm
bounded by $1$.

Assume first that $\frac{1}{p_1}+\frac{1}{p_2}\le 1$.  Hence $T$
is bounded in $L^1(\R)\times L^{p_1}(\R)\times L^{p_2}(\R) \to
L^{p}(\R)$ for $\frac{1}{p_1}+\frac{1}{p_2}= \frac{1}{p}$.

On the other hand, using H\"older's inequality
$$\sup_{x}|\int_\R
f(x-t)g(x+t)K(t)dt|\le \|f\|_{p_1}\|g\|_{p_2}\|K\|_{p'}.$$ This
shows that $T$ is also bounded in $L^{p'}(\R)\times
L^{p_1}(\R)\times L^{p_2}(\R) \to L^{\infty}(\R)$. Therefore, by
interpolation, selecting $0<\theta<1$ such that
$\frac{1}{p_3}=\frac{1-\theta}{p}$, one obtains that $T$ is
bounded in $L^{q}(\R)\times L^{p_1}(\R)\times L^{p_2}(\R) \to
L^{p_3}(\R)$ for
$\frac{1}{p_1}+\frac{1}{p_2}-\frac{1}{p_3}=1-\frac{1}{q}$.

Assume now that $1<\frac{1}{p_1}+\frac{1}{p_2}\le 2$.

Using that $\int_\R f(x-t)g(x+t)dt= f*g(2x)$, Young's inequality
implies that $$\|\int_\R f(x-t)g(x+t)K(t)dt\|_{r_3}\le
\|K\|_\infty \| D^\infty_{1/2}(|f|*|g|)\|_{r_3}\le
C\|f\|_{r_1}\|g\|_{r_2} \|K\|_\infty$$ whenever
$\frac{1}{r_1}+\frac{1}{r_2}\ge 1$ and
$\frac{1}{r_1}+\frac{1}{r_2}- 1=\frac{1}{r_3}$.

Hence $T$ is bounded in $L^{\infty}(\R)\times L^{p_1}(\R)\times
L^{p_2}(\R) \to L^{p}(\R)$ where
$\frac{1}{p_1}+\frac{1}{p_2}-1=\frac{1}{p}\le 1$.

Using duality, $\langle T(K,f,g), h\rangle = \langle T(h,\bar
f,g), K\rangle$, where $\bar f(x)=f(-x$, that is
$$\int_{\R^2} f(x-t)g(x+t)K(t)h(x) dtdx = \int_\R  (\int_\R \bar f(t-x)g(x+t)h(x) dx) K(t)dt.$$
Therefore $T$ is also bounded in $L^{p'}(\R)\times
L^{p_1}(\R)\times L^{p_2}(\R) \to L^{1}(\R)$.

Select $0\le\theta\le 1$ such that
$\frac{1}{p_3}=\frac{1}{p}+\frac{\theta}{p'}$. Now using
interpolation $T$ will be bounded in $L^{q}(\R)\times
L^{p_1}(\R)\times L^{p_2}(\R) \to L^{p_3}(\R)$ for
$\frac{1}{q}=\frac{\theta}{p'}=\frac{1}{p_3}-\frac{1}{p}=\frac{1}{p_3}-\frac{1}{p_1}-\frac{1}{p_2}+1$.

\end{proof}

\end{document}